\documentclass[a4paper, 12pt]{article}

\usepackage[sort&compress]{natbib}
\bibpunct{(}{)}{;}{a}{}{,} 

\usepackage{amsthm, amsmath, amssymb, mathrsfs, multirow, url, subfigure}
\usepackage{graphicx} 
\usepackage{ifthen} 
\usepackage{amsfonts}
\usepackage[usenames]{color}
\usepackage{fullpage}
\theoremstyle{plain} 
\newtheorem{theorem}{Theorem}
\newtheorem{cor}{Corollary}

\newtheorem{lemma}{Lemma}

\theoremstyle{definition}

\theoremstyle{remark}
\newtheorem{remark}{Remark}

\newcommand{\RR}{\mathbb{R}}

\newcommand{\E}{\mathsf{E}}

\newcommand{\prob}{\mathsf{P}}

\renewcommand{\phi}{\varphi}

\renewcommand{\S}{\mathcal{S}}

\newcommand{\nm}{\mathsf{N}}

\newcommand{\bin}{\mathsf{Bin}}

\newcommand{\unif}{\mathsf{Unif}}

\title{Empirical priors and coverage of posterior credible sets in a sparse normal mean model}
\author{
Ryan Martin\footnote{Department of Statistics, North Carolina State University, {\tt rgmarti3@ncsu.edu}} \quad and \quad Bo Ning\footnote{Department of Statistics and Data Science, Yale University, {\tt bo.ning@yale.edu}}
}
\date{\today}

\begin{document}

\maketitle 

%\doublspacing

\begin{abstract}
Bayesian methods provide a natural means for uncertainty quantification, that is, credible sets can be easily obtained from the posterior distribution.  But is this uncertainty quantification valid in the sense that the posterior credible sets attain the nominal frequentist coverage probability?  This paper investigates the frequentist validity of posterior uncertainty quantification based on a class of empirical priors in the sparse normal mean model.  In particular, we show that our marginal posterior credible intervals achieve the nominal frequentist coverage probability under conditions slightly weaker than needed for selection consistency and a Bernstein--von Mises theorem for the full posterior, and numerical investigations suggest that our empirical Bayes method has superior frequentist coverage probability properties compared to other fully Bayes methods.  

%We prove that there are scenarios in which the empirical Bayes method provides valid uncertainty quantification while other methods may not, and finite-sample simulations confirm the asymptotic findings.

\smallskip

\end{abstract}

\section{Introduction}
\label{S:intro}

Consider the high-dimensional normal mean model 
%\begin{equation}
%\label{eq:model}
\[ Y_i \sim \nm(\theta_i, \sigma^2), \quad i=1,\ldots,n, \]
%\end{equation}
where $Y=(Y_1,\ldots,Y_n)^\top$ are independent, the variance $\sigma^2 > 0$ is known, and inference on the unknown mean vector $\theta = (\theta_1,\ldots,\theta_n)^\top$ is desired.  We call this ``high-dimensional'' because there are $n$ unknown parameters but only $n$ data points, with $n \to \infty$.  It is virtually hopeless to make quality inference on a high-dimensional parameter without some additional structure to effectively reduce the dimension, and the structure we assume here is {\em sparsity}, i.e., most $\theta_i$'s are zero.  Work on the sparse normal mean problem goes back at least to \citet{donohojohnstone1994b}, and it is now a canonical example in the high-dimensional inference literature.  Some recent efforts on constructing estimators that achieve the asymptotically optimal minimax rate can be found in \citet{johnstonesilverman2004}, \citet{jiang.zhang.2009}, \citet{castillo.vaart.2012}, \citet{martin.walker.eb}, \citet{pas.kleijn.vaart.2014, pas.szabo.vaart.rate}, and \citet{ghosh.chakrabarti.shrink}.  Extensions to the high-dimensional regression setting can be found in \citet{castillo.schmidt.vaart.reg}, \citet{martin.mess.walker.eb}, and \citet{ning.ghosal}.  

The first author (RM) learned about and came to appreciate the intricacies of this sparse normal means problem from his primary advisor, Professor Jayanta K.~Ghosh, or JKG for short.  His book \citep[][Chapter~9]{ghosh-etal-book} gives an excellent introduction to this problem, including the classical James--Stein and full/empirical Bayes solutions, and much of JKG's later work was devoted to the large-scale significance testing version of this problem, e.g., \citet{bogdan.ghosh.tokdar.2008}, \citet{bcfg2010}, and \citet{datta.ghosh.2013}, among others, and there have been many subsequent developments based on his work.  Sadly, JKG passed away on September 30th, 2017, before we had a chance to talk with him about our empirical priors. But it is an honor present some new developments on one of JKG's favorite problems in this paper dedicated to him.

An advantage of a Bayesian approach to these problems is that it yields a posterior distribution which can be used to quantify uncertainty.  Computational issues aside, it is easy to construct a posterior credible set for $\theta$.  The question is: does the credibility level of the aforementioned set equal the frequentist coverage probability of the set?  If so, then we say that the posterior provides {\em valid uncertainty quantification}.  It is not clear, however, whether one should expect this property to hold in problems where the dimension is increasing with the sample size.  In fact, there are negative results \citep[e.g.,][]{li1989} which say, roughly, credible sets based on posterior distributions that adaptively achieve the optimal concentration rate cannot simultaneously provide valid uncertainty quantification uniformly over the parameter space.  That is, for any posterior with the optimal concentration rate, there must be true parameter values that the $100(1-\gamma)$\% posterior credible sets cover with probability (much) less than $1-\gamma$.  Consequently, there is interest in identifying these troublemaker parameter values.  Recent efforts along these lines in the sparse normal mean model include \citet{belitser.ddm}, \citet{belitser.nurushev.uq}, \citet{pas.szabo.vaart.uq}, \citet{castillo.szabo.coverage}, and \citet{nurushev.belitser.projection}; for details beyond the normal mean model, see, e.g., \citet{szabo.vaart.zanten.2015}, \citet{belitser.ghosal.ebuq}, and \citet{ebpred}.  

In this paper we pursue the question of valid posterior uncertainty quantification in the sparse normal mean model in several new ways.  First, we focus our investigation on a relatively new type of posterior distribution based on a suitable {\em empirical prior}.  This approach differs considerably from classical empirical Bayes and has been shown to have strikingly good practical and theoretical performance; see \citet{martin.walker.deb} and the references in Section~\ref{S:model}.  That this method would also provide valid uncertainty quantification was conjectured by \citet{martin.horseshoe.discuss} and our work here confirms that.  Second, while previous investigations into posterior uncertainty quantification have focused primarily on credible $\ell_2$-balls for the full $\theta$ vector, we consider an arguably more practical question of marginal credible intervals for certain features of $\theta$.  From a general Bernstein--von Mises result for the full posterior in Theorem~\ref{thm:bvm}, we give sufficient conditions for valid uncertainty quantification about a general linear functional of $\theta$ in Corollary~\ref{cor:coverage}.  For a specific linear functional, however, such as $\theta_1$, one might ask if the conditions can be weakened.  To this end, we show that the nominal coverage can be achieved even if the true $|\theta_1^\star|$ is slightly smaller than what is needed for selection consistency or the Bernstein--von Mises theorem; see Theorem~\ref{thm:medium}.  To our knowledge, there are no results---at least none that are positive---concerning the coverage probability properties of marginal credible intervals for the horseshoe or any other prior distributions.  Third, we give numerical results to compare the performance of credible sets based on our empirical priors and those based on the horseshoe prior.  These results confirm what the theory suggests, namely, that there is a non-trivial range of non-zero $\theta_1^\star$ values that the former can properly cover while the latter cannot.  Finally, in Section~\ref{S:discuss}, we provide some concluding remarks and directions for future research.

\section{An empirical Bayes model}
\label{S:model}

\subsection{Prior and posterior construction}
\label{SS:prior}

Under assumption of sparsity, it makes sense to re-express the vector $\theta$ as a pair $(S, \theta_S)$, where $S \subseteq \{1,2,\ldots,n\}$, the {\em configuration} of $\theta$, indicates which coordinates are non-zero, and $\theta_S$ is an $|S|$-vector that contains $S$-specific non-zero values; here $|S|$ is the cardinality of $S$.  With this decomposition, a hierarchical prior is natural, i.e., a prior for the configuration $S$ and then, given $S$, a prior for the $S$-specific parameter $\theta_S$.  The sparsity assumption provides some relevant prior information to help with the choice of the prior on $S$.  On the other hand, no prior information is available for $\theta_S$, so a ``default'' prior is recommended.  It turns out that certain features of this default prior---in particular, its tails---can significantly affect the properties of the corresponding posterior.  Indeed, using normal priors for $\theta_S$ can lead to suboptimal posterior concentration properties, whereas heavier-tailed priors like Laplace have much better performance \citep[e.g.,][]{castillo.vaart.2012}.  Unfortunately, while the normal priors are conjugate and lead to relatively simple posterior computations, the heavy-tailed priors make computation more difficult.  So one faces a dilemma: use a simple conjugate prior that is easy to compute but may have suboptimal theoretical performance, or pay the non-trivial computational price for the use of a theoretically justified prior.  Is such a choice really necessary?  

\citet{martin.walker.eb} argued that the prior tails cannot have much of an effect if its center is appropriately chosen.  In other words, an appropriately-centered conjugate normal prior should not suffer the same suboptimality property as if it had a fixed center.  Of course, an ``appropriate'' center can only originate from the data, hence an {\em empirical prior}.  The particular form of our empirical prior, $\Pi_n$, for $\theta=(S,\theta_S)$ is 
\[ S \sim \pi(S) \quad \text{and} \quad (\theta_S \mid S) \sim \nm_{|S|}(Y_S, \sigma^2 \tau^{-1} I_{|S|}), \]
where $\pi(S)$ is a marginal prior for the configuration $S$ to be specified below, $Y_S$ is the sub-vector of $Y$ corresponding to configuration $S$, and $\tau > 0$ is a prior precision factor which, as we discuss later, will be taken relatively small.  Again, the idea is to properly center the prior based on the data, so that the thin normal tails will not affect the asymptotic concentration properties.  
%We denote this empirical prior for $\theta$ as $\Pi_n$.  
A regression version of this formulation is given in \citet{martin.mess.walker.eb}, and \citet{martin.walker.deb} describe a general empirical prior framework.   

Before getting into the specifics of the empirical prior formulation, we pause here to say a few words about the philosophy behind such an approach.  As we see it, in modern Bayesian analysis, especially that of high-dimensional problems, the prior is treated primarily as an input that can be tuned and tweaked in order to obtain a posterior distribution which is ``good'' in one ore more senses.  For example, the horseshoe prior of \citet{carvalho.polson.scott.2010} and its various extensions are popular nowadays, not because they capture genuine subjective beliefs, but because the corresponding posterior is (believed to be) fast to compute and achieves desirable asymptotic concentration rates and uncertainty quantification properties.\footnote{In fact, it is not uncommon in seminar talks or less formal discussions to hear one motivate the construction of a new prior by saying that existing priors ``don't work'' and/or the new prior ``works better,'' not that it more accurately reflects subjective prior beliefs, etc.}  If the prior is chosen primarily for the frequentist properties that the corresponding posterior satisfies, then there is no reason not to consider using a data-dependent prior if that too can produce a posterior with as good or better properties.  A goal of this paper, and the related literature on empirical priors, is to show that, indeed, they achieve these desirable properties and, at least in some cases, outperform the state-of-the-art Bayesian methods based on fixed priors.  

Returning to the construction, one has some flexibility in the choice of marginal prior for the configuration and, in our numerical investigations in Section~\ref{S:examples}, we will consider two such priors.  Both decompose the prior $\pi$ for $S$ into a marginal prior, $f_n(s)$, for the size, $s=|S|$, of the configuration, and a uniform prior over the space of configurations of a given size.  Below are the respective mass functions, $f_n$, for the two priors.  
\begin{itemize}
\item The {\em complexity prior}, described in \citet{castillo.vaart.2012} and used in \citet{martin.mess.walker.eb} for regression, has mass function $f_n$ given by 
\begin{equation}
\label{eq:complexity}
f_n(s) \propto (c n^a)^{-s}, \quad s=0,1,\ldots,n, 
\end{equation}
for constants $a, c > 0$.  This is just a truncated geometric prior with success probability equal to $1 - (cn^a)^{-1}$ so, depending on $a$, it strongly penalizes complex configurations, i.e., those models with a large number of non-zero means.  
\vspace{-2mm}
\item The {\em beta--binomial prior} in \citet{martin.walker.eb} is of the form 
\begin{equation}
\label{eq:beta.binom}
f_n(s) = b_n \binom{n}{s} \int_0^1 w^{n + b_n - s - 1} (1-w)^s \, dw, 
\end{equation}
where $b_n \to \infty$ as $n \to \infty$.  This corresponds to a beta prior, ${\sf Beta}(b_n,1)$, on a latent variable $W$ and, given $W=w$, a binomial prior, $\bin(n, 1-w)$, for $|S|$.  The mean of $W$ is $b_n(b_n + 1)^{-1}$, which makes the prior mean of $|S|$ close to zero so, like \eqref{eq:complexity}, the prior in \eqref{eq:beta.binom} also severely penalizes complex configurations.  
\end{itemize}

Finally, the empirical Bayes posterior distribution for $\theta$, denoted by $\Pi^n$, is obtained by combining the prior with a fractional power of the likelihood function according to Bayes's theorem.  That is, 
\begin{equation}
\label{eq:post}
\Pi^n(d\theta) \propto L_n(\theta)^\alpha \, \Pi_n(d\theta), 
\end{equation}
where $\alpha \in (0,1)$ and $L_n(\theta) \propto  \exp\{-\frac{1}{2\sigma^2} \|Y-\theta\|^2\}$, with $\|\cdot\|$ the $\ell_2$-norm.  The power $\alpha$ does not affect the normality of the likelihood, $L_n(\theta)$, it only inflates the variance, hence making its contours wider.  In fact, given the normal form of the conditional priors, the posterior actually takes a relatively simple form; see Section~\ref{SS:computation} below.  

As for the fractional power, the results in \citet{belitser.nurushev.uq} and \citet{belitser.ghosal.ebuq} suggest that taking $\alpha=1$ might be possible, perhaps with some adjustments elsewhere, but we believe there are reasons to retain this flexibility, especially in the context of uncertainty quantification.  In particular, existing results on coverage probability of credible sets require a {\em blow-up factor}---e.g., the factor $L$ in Equation~(3) of \citet{pas.szabo.vaart.uq}---to inflate the credible set beyond the size determined by the posterior distribution itself.  As is well-known in the generalized Bayes community \citep[e.g.,][]{grunwald.ommen.scaling, syring.martin.scaling, holmes.walker.scaling}, the inclusion of $\alpha$ makes the posterior distribution wider and, therefore, also has a beneficial blow-up effect on the credible sets; see Remark~\ref{re:alpha}.

\subsection{Posterior computation}
\label{SS:computation}

As advertised, the empirical prior leads to relatively simple posterior computation.  In general, since the prior for $\theta_S$, given $S$, is normal, it is possible to analytically integrate out $\theta_S$ to obtain a marginal posterior distribution\footnote{The expression for $\pi^n(S)$ given in Section~4.1 of \citet{martin.mess.walker.eb} has a typo, but the correct formula is given in the supplement at \url{https://arxiv.org/abs/1406.7718}.} for the configuration $S$, i.e., 
\begin{equation}
\label{eq:S.post}
\pi^n(S) \propto \pi(S) (1 + \alpha \tau^{-1} )^{-|S|/2} e^{-\frac{\alpha}{2\sigma^2} \|Y_{S^c}\|^2}. 
\end{equation}
In situations where only characteristics of the marginal posterior of $S$ are needed, e.g., for feature selection, certain credible sets (see Section~\ref{S:coverage}), etc, then the above formula can be used to design a straightforward Metropolis--Hastings algorithm to sample $S$; \citet{ecap} adopt a shotgun stochastic search based on \eqref{eq:S.post}.  Since $\theta_S$, given $S$, is normal, the Metropolis--Hastings procedure can be easily augmented to sample both $S$ and $\theta_S$.  

With the beta--binomial prior, there is additional structure that can be used to design a posterior sampling algorithm.  That is, by introducing that latent variable ``$W$,'' the full conditionals---$\theta \mid W$ and $W \mid \theta$---are available in closed-form, hence a Gibbs sampler, as in \citet{martin.walker.eb}, can be readily employed.

\subsection{Asymptotic concentration properties}
\label{SS:concentration}

Let $\theta^\star \in \RR^n$ denote the true mean vector, with configuration $S^\star = S_{\theta^\star}$, where the notation ``$S_\theta$'' refers to the configuration of the vector $\theta$.  We assume that $\theta^\star$ is sparse and, in particular, $s^\star = |S^\star|$ is $o(n)$ as $n \to \infty$.  The results reviewed below show that the posterior, $\Pi^n$, is able to optimally identify both $\theta^\star$ and $S^\star$ as $n \to \infty$.  

The only aspect of the model described above that is not completely determined is the prior $f_n$ on the configuration size.  The key is that the tails of $f_n$ have a certain rate of decay, as described in Equation~(2.2) of \citet{castillo.vaart.2012} and Assumption~1 of \citet{castillo.schmidt.vaart.reg}.  In particular, we assume here that 
\begin{equation}
\label{eq:fn}
K_1 n^{-a_1} \leq \frac{f_n(s)}{f_n(s-1)} \leq K_2 n^{-a_2}, \quad s=1,\ldots,n, 
\end{equation}
where $(K_1,K_2)$ and $(a_1,a_2)$ are suitable constants; in particular $a_1 > a_2$.  This implies that $f_n(s+t)/f_n(s)$ is lower and upper bounded by $(K_jn^{-a_j})^t$ for $j=1,2$, respectively, which is critical for the existing methods of proof.  Clearly, the complexity prior \eqref{eq:complexity} satisfies this and, moreover, any prior that does so must be similar to that in \eqref{eq:complexity}.  It turns out that the beta--binomial prior also satisfies \eqref{eq:fn} if $b_n \propto n^\xi$ for some $\xi > 1$.  

\begin{theorem}
\label{thm:rate}
If the prior for the configuration size satisfies \eqref{eq:fn}, then the posterior $\Pi^n$ adaptively attains the minimax concentration rate, i.e., for a constant $M' > 0$, 
\[ \sup_{\theta^\star: |S_{\theta^\star}| = s^\star} \E_{\theta^\star} \Pi^n(\{\theta \in \RR^n: \|\theta-\theta^\star\|^2 > M' s^\star \log(n / s^\star)\}) \to 0, \quad n \to \infty. \]
\end{theorem}

\begin{proof}
Follows from the arguments in \citet{martin.mess.walker.eb}.  
\end{proof}

Besides concentration of the posterior around $\theta^\star$, it is interesting to consider the posterior for the configuration.  Ideally, it would put all of its mass on $S^\star$, asymptotically, but this requires some conditions on the magnitudes of the non-zero means; see Result~3 in Theorem~\ref{thm:selection} below.  These selection results will be used in Section~\ref{S:coverage}.  
%Concentration properties of the marginal posterior for $S$ will be important to our work on coverage of posterior credible sets in Section~\ref{S:coverage}.  

\begin{theorem}
\label{thm:selection}
For the pair $a_1 > a_2$ in \eqref{eq:fn}, assume that $s^\star n^{-a_2} \to 0$
\begin{enumerate}
\item {\em No proper supersets.} $\E_{\theta^\star} \Pi^n(\{\theta: S_\theta \supset S^\star\}) \to 0$.
\item {\em No important misses.} Define the threshold
\begin{equation}
\label{eq:bound}
\rho_n = \Bigl\{ \frac{2\sigma^2 (1+\alpha) M \log n}{\alpha} \Bigr\}^{1/2}
\end{equation}
where $M = \max\{M', 1+a_1\}$ for the constant $M'$ as in Theorem~\ref{thm:rate}.  Define $S^\dagger = \{i: |\theta_i^\star| > \rho_n\} \subseteq S^\star$; then $\E_{\theta^\star} \Pi^n(\{\theta: S_\theta \not \supseteq S^\dagger\}) \to 0$.  
\item {\em Selection consistency.} If $S^\dagger=S^\star$, i.e., if all the non-zero means have magnitude larger than $\rho_n$ in \eqref{eq:bound}, then $\E_{\theta^\star} \pi^n(S^\star) \to 1$.  
\end{enumerate}
\end{theorem}

\begin{proof}
Part~1 is a consequence of the analysis in \citet{martin.mess.walker.eb}; a proof of Part~2 is given in Appendix~\ref{S:proofs} below.  Part~3 follows immediately from Parts~1--2.
\end{proof}

\section{Uncertainty quantification}
\label{S:coverage}

\subsection{Bernstein--von Mises phenomenon}

The key to proving that the posterior uncertainty quantification is valid is establishing some suitable approximation of the posterior in terms of  familiar quantities whose distribution is known.  A Bernstein--von Mises theorem, for example, would show that the posterior is approximately normal.  Therefore, posterior credible sets must closely resemble the corresponding sets for the normal approximation, and properties of the mean and variance of the normal distribution can be used to obtain coverage results.  In our present case, under suitable conditions, a Bernstein--von Mises theorem is relatively easy to establish, and it says that the posterior $\Pi^n$ can be approximated, asymptotically, by a product of normal distribution on the true configuration, $S^\star$, and a point mass at 0 on its complement, $S^{\star c}$.  Compare this result to Theorem~6 in \citet{castillo.schmidt.vaart.reg}.  

\begin{theorem}
\label{thm:bvm}
Let $d_{\text{\sc tv}}$ denote the total variation distance and $S^\star = S_{\theta^\star}$.  Then 
\begin{equation}
\label{eq:bvm}
\E_{\theta^\star} d_{\text{\sc tv}}\bigl( \Pi^n, \nm_{|S^\star|}(Y_{S^\star}, v_\alpha I_{|S^\star|}) \otimes \delta_{S^{\star c}} \bigr) \to 0
\end{equation}
for all $\theta^\star$ such that $\E_{\theta^\star} \pi^n(S^\star) \to 1$, where $v_\alpha = \sigma^2(\alpha + \tau)^{-1}$ and $\delta_{S^{\star c}}$ denotes the point mass distribution for $\theta_{S^{\star c}}$ concentrated at the origin.  
\end{theorem}

\begin{proof}
Let $D_n=D_n(Y)$ denote the total variation distance in \eqref{eq:bvm}.  By definition, 
\begin{equation}
\label{eq:tv}
D_n = \sup_A \bigl| \Pi^n(A) - \{\nm_{|S^\star|}(Y_{S^\star}, v_\alpha I_{|S^\star|}) \otimes \delta_{S^{\star c}}\}(A) \bigr|, 
\end{equation}
where the supremum is over all Borel measurable subsets of $\RR^n$.  Next, again by definition, 
\[ \Pi^n(A) = \sum_S \pi^n(S) \{\nm_{|S^\star|}(Y_{S^\star}, v_\alpha I_{|S^\star|}) \otimes \delta_{S^{\star c}}\}(A), \]
so we can immediately upper-bound the absolute value on the right-hand side of \eqref{eq:tv} by 
\[ \sum_S \pi^n(S) \bigl| \{\nm_{|S|}(Y_S, v_\alpha I_{|S|}) \otimes \delta_{S^c}\}(A) - \{\nm_{|S^\star|}(Y_{S^\star}, v_\alpha I_{|S^\star|}) \otimes \delta_{S^{\star c}}\}(A) \bigr|. \]
The absolute difference above is 0 when $S=S^\star$ and bounded by 2 otherwise, so 
\[ D_n \leq 2 \sum_{S \neq S^\star} \pi^n(S) = 2 \{1 - \pi^n(S^\star)\}. \]
After taking expectation of both sides, if $\theta^\star$ is such that $\E_{\theta^\star} \pi^n(S^\star) \to 1$, then the upper bound vanishes, thus proving the theorem.
%Bernstein--von Mises phenomenon, \eqref{eq:bvm}.  For the second claim, convergence in total variation implies the posterior quantile, $t_\gamma$, for $\Pi_\psi^n$ asymptotically agrees with the corresponding quantile for $\nm(\hat\psi_{S^\star}, v_\alpha\|x_{S^\star}\|^2)$.  The latter quantile is familiar, and the claim that \eqref{eq:bvm} implies \eqref{eq:coverage} follows immediately.  
\end{proof}

\begin{remark}
\label{re:alpha}
Recall that $v_\alpha = \sigma^2(\alpha + \tau)^{-1}$.  Therefore, in order for the variance in the posterior distribution to be at least as large as the variance of $Y_{S^\star}$, we need $v_\alpha \geq \sigma^2$ or, equivalently, $\alpha + \tau \leq 1$.  The other properties of the posterior require $\alpha < 1$ and $\tau > 0$ so we do not have the option to choose $\alpha=1$ and $\tau=0$ at this point.  We do have the option to take $\tau=1-\alpha$, in which case $v_\alpha = \sigma^2$.  However, in the spirit of conservatism and validity, we prefer to take $\alpha + \tau$ slightly less than 1.  This agrees with \citet{grunwald.mehta.rates} who say that Bayesian inference is ``safe'' if the learning rate---their version of $\alpha$---is {\em strictly less than} a critical threshold.  In particular, for our simulations in Section~\ref{S:examples}, we take $\alpha=0.95$ and $\tau = 0.025$, so that $\alpha + \tau = 0.975 < 1$.
\end{remark}

%For uncertainty quantification, instead of credible balls for the full $\theta$ vectors, as has been investigated in {\color{blue} refs}, here we focus on an arguably more practical case of certain one-dimensional features of the full $\theta$ vector.  In particular, we consider a general linear functional $\psi = x^\top \theta$, where $x \in \RR^n$ is a specified vector.  This includes, as an important special case, inference about an individual mean, say, $\theta_1$, and we expect that other functionals of $\theta$, which are only approximately linear, could be handled in a similar way. 

\subsection{Credible sets and their coverage properties}

The literature on Bayesian uncertainty quantification in sparse high-dimensional problems has focused primarily on the coverage of $\ell_2$-balls centered at the posterior mean.  Here we consider a different and arguably more practical question concerning the coverage of marginal credible intervals for certain one-dimensional features of $\theta$.  In particular, we consider a general linear functional $\psi = x^\top \theta$, where $x \in \RR^n$ is a specified vector.  This includes, as an important special case, inference about an individual mean, say, $\theta_1$, and we expect that other functionals of $\theta$, which are only approximately linear, could likely be handled in a similar way.  

Without loss of generality, we consider coverage probability of upper credible bounds for $\psi$.  Towards this, we need its marginal posterior distribution function, call it $H_n(t) = \Pi^n(\{\theta: x^\top \theta \leq t\})$, which is given by 
\begin{equation}
\label{eq:mixture}
H_n(t) = \sum_S \pi^n(S) \, \Pi^n(\{\theta: x_S^\top \theta_S \leq t\} \mid S) = \sum_S \pi^n(S) \, \Phi\Bigl( \frac{t - \hat\psi_S}{v_\alpha^{1/2} \|x_S\|} \Bigr), 
\end{equation}
where $\Phi$ is the standard normal distribution function, $x_S$ is the $|S|$-vector corresponding to configuration $S$, $\hat\psi_S = x_S^\top Y_S$, and $v_\alpha = \sigma^2(\alpha + \tau)^{-1}$ as before.  

For the special case where $x=e_k$, the $k^\text{th}$ standard basis vector, this reduces to 
\begin{equation}
\label{eq:Hn}
H_n(t) = (1-p_k^n) \, 1_{[0,\infty)}(t) + p_k^n \, \Phi\bigl( v_\alpha^{-1/2}(t-Y_k) \bigr), 
\end{equation}
where $p_k^n = \pi^n(S \ni k)$ is the {\em posterior inclusion probability} of index $k$, i.e., the posterior probability that the configuration contains the particular index $k$, a quantity often used in Bayesian feature selection problems \citep[e.g.,][]{barbieri.berger.2004}.  The above expression reveals that the marginal posterior of $\theta_k$ is, as expected, a mixture of a point mass at zero and a normal distribution centered at $Y_k$.  The only setback is that the inclusion probability, $p_k^n$, is a complicated function of the full data.  

Towards uncertainty quantification, fix a nominal confidence level $\gamma \in (0,\frac12)$.  Then the $100(1-\gamma)$\% credible upper bound for $\psi=x^\top\theta$, denoted by $t_\gamma$, is given by 
\[ t_\gamma = \inf\{t: H_n(t) \geq 1-\gamma\}. \]
The somewhat complicated expression here is to handle the discontinuity in $H_n$ at the origin.  That is, if the inclusion probability is small, there will be a large jump at the origin, in which case, $H_n(t)=1-\gamma$ has no solution, but the infimum ensures credible upper bound will be $t_\gamma = 0$.  Otherwise, if there is no large jump at the origin, then $H_n(t) = 1-\gamma$ has a unique solution, and that would be $t_\gamma$.  Note that this credible bound is precisely the one that practitioners would read off from the corresponding quantiles of the posterior samples, different from the marginal posterior mean plus an inflated standard error as studied in \citet{pas.szabo.vaart.uq} and elsewhere.  

Of course, the credible bound $t_\gamma=t_\gamma(Y)$ depends on the data, so, for the posterior uncertainty quantification to be valid in a frequentist sense, the relevant property is that the coverage probability is close to the nominal $1-\gamma$ level, i.e., 
\begin{equation}
\label{eq:coverage}
\prob_{\theta^\star}\bigl( t_\gamma \geq x^\top \theta^\star \bigr) \geq 1-\gamma + o(1), \quad n \to \infty. 
\end{equation}
It turns out that this coverage probability property is a more-or-less immediate consequence of the Bernstein--von Mises phenomenon.

%To get some intuition about why one would expect the above property to hold, consider the special case of a marginal credible bound for $\theta_k$.  If $\theta_k^\star$ is actually zero, then the inclusion probability should vanish, making the asymptotic posterior distribution of $\theta_k$ a point mass at zero, hence \eqref{eq:coverage} holds trivially.  If $|\theta_k^\star|$ is large, then the inclusion probability should go to 1, hence the marginal posterior distribution for $\theta_k$ is roughly $\nm(Y_k, v_\alpha)$ and, if $v_\alpha \geq \sigma^2$, then \eqref{eq:coverage} holds again.  The case of a moderate non-zero $|\theta_k^\star|$ is more difficult to assess, but it all rests on the asymptotic properties of $\pi^n(S^\star)$.  

\begin{cor}
\label{cor:coverage}
If $\alpha + \tau \leq 1$, so that $v_\alpha \geq \sigma^2$ then \eqref{eq:bvm} implies \eqref{eq:coverage} and, therefore, valid uncertainty quantification. 
\end{cor} 

\begin{proof}
Intuitively, if $\theta \sim \Pi^n$ is approximately normal, then $\psi=x^\top \theta$ is approximately normal too, with mean $\hat\psi_{S^\star}$ and variance $v_\alpha \|x_{S^\star}\|^2$.  More formally, an argument identical to that in the proof of Theorem~\ref{thm:bvm} gives 
\[ \E_{\theta^\star} d_{\text{\sc tv}}\bigl( \Pi_\psi^n, \nm(\hat\psi_{S^\star}, v_\alpha \|x_{S^\star}\|^2) \bigr) \to 0, \]
where $\Pi_\psi^n$ is the derived posterior distribution of $\psi = x^\top \theta$.  This convergence in total variation implies that the posterior quantile, $t_\gamma$, asymptotically agrees with the corresponding quantile for $\nm(\hat\psi_{S^\star}, v_\alpha \|x_{S^\star}\|^2)$.  The latter quantile and its distributional properties are familiar and, therefore, \eqref{eq:bvm} implies \eqref{eq:coverage} if $v_\alpha \geq 1$. 
\end{proof}

To summarize so far, if the posterior distribution can correctly identify $S^\star$, the true configuration, i.e., if $\pi^n(S^\star) \to 1$, then the corresponding posterior uncertainty quantification about any linear functional will be valid in a frequentist sense.  We know, from Theorem~\ref{thm:selection}, that if all the non-zero means have magnitude exceeding the threshold $\rho_n$ in \eqref{eq:bound}, then $\pi^n(S^\star) \to 1$.  But it is worth asking if this can be established under weaker conditions, at least in certain cases.  One such case, which is practically relevant, is a marginal credible interval for $\psi=\theta_k$, say, and we ask how large does $|\theta_k^\star|$ need to be in order for the marginal posterior uncertainty quantification to be valid.  
%Different scenarios could be considered, but here we will assume that each non-zero entry in $\theta^\star$ has magnitude exceeding the threshold $\rho_n$ in \eqref{eq:bound} except for $\theta_k^\star$.  
This is directly related to the numerical experiments considered in Section~\ref{S:examples} below.  

At least intuitively, for the marginal posterior for $\psi=\theta_k$, whether the uncertainty quantification is valid hinges entirely on the behavior of the inclusion probability.  For example, if $\theta_k^\star \neq 0$, and if $p_k^n \to 1$, then the coverage probability of the $100(1-\gamma)$\% marginal posterior credible upper bound should be approximately $1-\gamma$.  The following theorem confirms this intuition.   

\begin{lemma}
\label{lem:incl.prob}
The $100(1-\gamma)$\% marginal posterior credible interval for $\theta_k$ is valid, i.e., the frequentist coverage probability is at least $1-\gamma$ for sufficiently large $n$, if $\theta_k^\star = 0$ and $\E_{\theta^\star}(p_k^n) \to 0$ or if $\theta_k^\star \neq 0$ and $\E_{\theta^\star}(1-p_k^n) \to 0$.  
\end{lemma}

\begin{proof}
Start with the case $\theta_k^\star=0$.  If $t_\gamma$ is the corresponding upper credible limit, then the relevant coverage probability can be bounded as follows:
\begin{align*}
\prob_{\theta^\star}(t_\gamma \geq 0) & = \prob_{\theta^\star}\{H_n(0) \geq 1-\gamma) \\
& \geq \prob_{\theta^\star}(1 - p_k^n \geq 1-\gamma) \\
& \geq \gamma^{-1} \{ \E_{\theta^\star}(1-p_k^n) - (1-\gamma)\} \\
& = \gamma^{-1}\{ \gamma - \E_{\theta^\star}(p_k^n) \}.
\end{align*}
The first inequality above is due to the fact that $H_n(0) \geq 1-p_k^n$; the second is by the reverse Markov inequality.  By assumption, $\E_{\theta^\star}(p_k^n) \to 0$ and, therefore, the coverage probability converges to 1, as expected.  

Next, without loss of generality, we consider $\theta_k^\star > 0$.  In this case, the relevant coverage event is $\{H_n(\theta_k^\star) \leq 1-\gamma\}$ where, in this case, 
\[ H_n(\theta_k^\star) = (1-p_k^n) + p_k^n \, V_k, \]
with $V_k = \Phi\{ v_\alpha^{-1/2}(Y_k - \theta_k^\star)\}$.  Note, if $\alpha + \tau \leq 1$, then $v_\alpha \geq \sigma^2$, which implies that  $V_k$ has tails no thinner than $\unif(0,1)$, the uniform distribution between 0 and 1.  Below we show that the non-coverage probability is upper-bounded by $\gamma$ as $n \to \infty$.  

The non-coverage probability is 
\[ \prob_{\theta^\star}\{H_n(\theta_k^\star) > 1-\gamma\} = \prob_{\theta^\star}\{ p_k^n(1-V_k) < \gamma\}. \]
Intuitively, since $p_k^n$ will be close to 1, and $1-V_k$ has tails no heavier than $\unif(0,1)$, we expect the right-most probability above to be less than $\gamma$.  To make this formal, introduce a $b \in (0,1)$ and write 
\begin{align*}
\prob_{\theta^\star}\{ p_k^n (1-V_k) \leq \gamma \} & \leq \prob_{\theta^\star}\{1-V_k < \tfrac{\gamma}{b} \text{ and } p_k^n > b\} + \prob_{\theta^\star}\{ p_k^n \leq b\} \\
& \leq \prob_{\theta^\star}\{1-V_k < \tfrac{\gamma}{b}\} + \prob_{\theta^\star}\{ p_k^n \leq b\} \\
& \leq \tfrac{\gamma}{b} + \prob_{\theta^\star}\{p_k^n \leq b\}.
\end{align*}
By Markov's inequality, 
\[ \prob_{\theta^\star}\{p_k^n \leq b\} = \prob_{\theta^\star}\{1 - p_k^n \geq 1-b\} \leq (1-b)^{-1} \E_{\theta^\star}(1-p_k^n). \]
If $\E_{\theta^\star}(1-p_k^n) \to 0$, as we assumed, then we get 
\[ \limsup_{n \to \infty} \, \bigl[ \text{non-coverage probability} \bigr] \leq \tfrac{\gamma}{b}. \]
Since this holds for all $b \in (0,1)$, it holds for the infimum over $b$ so, therefore, the limiting non-coverage probability is no bigger than $\gamma$, which proves the claim. 
\end{proof}

To control the inclusion probability in the case $\theta_k^\star \neq 0$, it will be necessary to distinguish the means whose magnitudes exceed the cutoff $\rho_n$ in \eqref{eq:bound} and those that are just simply non-zero, so recall, from Section~\ref{SS:concentration}, $S^\dagger = \{i: |\theta_i^\star| > \rho_n\}$ and let $s^\dagger = |S^\dagger|$.  Here we are specifically considering the possibility that $|\theta_k^\star|$ is less than $\rho_n$, so we have $s^\dagger < s^\star$.  The following theorem gives a lower bound on the magnitude of $\theta_k^\star$ such that its inclusion probability, $p_k^n$, converges to 1, and that bound depends on how large $s^+$ is compared to $s^\star$.  And, in light of Lemma~\ref{lem:incl.prob}, this establishes how large the mean needs to be in order for the marginal posterior distribution to provide valid uncertainty quantification.  

\begin{theorem}
\label{thm:medium}
If the conditions of Theorem~\ref{thm:selection} hold, and if $|\theta_k^\star| > \zeta_n$, where 
\begin{equation}
\label{eq:zeta.n}
\zeta_n^2 = \tfrac{2\sigma^2(1+\alpha)}{\alpha} \bigl\{ a_1 \log n + \log\bigl(\tfrac{n-s^\dagger}{s^\dagger + 1}\bigr) + (s^\star - s^\dagger - 1) \log 2 \bigr\}, 
\end{equation}
and $s^\dagger < s^\star$ is the number of non-zero means have magnitude exceeding the threshold in \eqref{eq:bound}, then the inclusion probability satisfies $\E_{\theta^\star}(p_k^n) \to 1$ as $n \to \infty$ and, consequently, for $\alpha + \tau \leq 1$, the marginal posterior for $\theta_k$ provides valid uncertainty quantification.  
\end{theorem}

\begin{proof}
Without loss of generality, take $\sigma^2=1$.  By Theorem~\ref{thm:selection}, we know that those $s^\dagger$ means are sufficiently large that the posterior will not exclude them.  Moreover, we know that no non-zero means will be included.  So the only configurations that need to be considered are those that contain all of the $s^\dagger$ large means and any of the $s^\star-s^\dagger$ non-zero but not-large means.  Let $\S^0$ denote this subset of configurations.  Then the exclusion probability for $\theta_k$ is asymptotically equal to  
\[ 1-p_k^n = \pi^n(S \not\ni k) = \sum_{S \in \S^0: S \not\ni k} \pi^n(S). \]
For a generic $S$ in this sum, we can bound the posterior probability, 
\[ \pi^n(S) \leq \frac{\pi^n(S)}{\pi^n(S \cup \{k\})} = \frac{\pi(S)}{\pi(S \cup \{k\})} (1 + \alpha \tau^{-1})^{1/2} e^{-\frac{\alpha}{2}|Y_k|^2}, \]
where the equality is based on plugging in the expression for $\pi^n$ in \eqref{eq:S.post}.  Taking expectation and using the bound on $\theta_k^\star$ in the lemma's statement, we get 
\[ \E_{\theta^\star} \{ \pi^n(S) \} \lesssim \frac{\pi(S)}{\pi(S \cup \{k\})} e^{-\frac{\alpha}{2(1+\alpha)} \zeta_n^2}. \]
Plugging in the definition of the prior $\pi$ for $S$, using the assumed bound in \eqref{eq:fn}, and writing $s=|S|$, we can simplify the right-hand side above:
\[ n^{a_1} \, \Bigl( \frac{n-s}{s+1} \Bigr) \, e^{-\frac{\alpha}{2(1+\alpha)} \zeta_n^2}. \]
Summing over all $S \in S^+$ with $S \not\ni k$, we get the upper bound 
\[ \E_{\theta^\star}(1 - p_k^n) = \sum_{S \in \S^+: S \not\ni k} \E_{\theta^\star}\{\pi^n(S)\} \lesssim n^{a_1} e^{-\frac{\alpha}{2(1+\alpha)} \zeta_n^2} \sum_{s=s^\dagger}^{s^\star-1} \binom{s^\star - s^\dagger - 1}{s-s^\dagger} \Bigl( \frac{n-s}{s+1} \Bigr). \]
It is easy to check that the remaining sum is upper-bounded by 
\[ \Bigl( \frac{n-s^\dagger}{s^\dagger+1} \Bigr) \, 2^{s^\star - s^\dagger - 1}. \]
Finally, we get 
\[ \E_{\theta^\star}(1 - p_k^n) \lesssim n^{a_1} \, e^{-\frac{\alpha}{2(1+\alpha)} \zeta_n^2} \, \Bigl(\frac{n-s^\dagger}{s^\dagger+1}\Bigr) \, 2^{s^\star - s^\dagger - 1}, \]
and it can be immediately seen that if $\zeta_n$ is as in \eqref{eq:zeta.n}, then the exclusion probability for $\theta_k$ vanishes.  Moreover, the claimed frequentist validity of the marginal posterior credible interval follows from Lemma~\ref{lem:incl.prob}.    
\end{proof}

One special case of the above setting is where all the $s^\star$ non-zero means exceed the threshold in \eqref{eq:bound} except for $\theta_k^\star$, so that $s^\dagger=s^\star-1$.  In that case, the threshold \eqref{eq:zeta.n} is 
\[ \zeta_n^2 = \tfrac{2\sigma^2(1+\alpha)}{\alpha} \bigl\{ a_1 \log n + \log\bigl(\tfrac{n-s^\star + 1}{s^\star}\bigr) \bigr\}. \]
Since $\log(\tfrac{n-s^\star+1}{s^\star})$ is strictly less than $\log n$, and the $M$ in \eqref{eq:bound} is larger than $1+a_1$, there is some improvement in the sense that $|\theta_k^\star|$ can actually be smaller than the threshold from Theorem~\ref{thm:selection} and still achieve the desired coverage probability.  However, since $s^\star=o(n)$, the difference is not substantial enough to impact the threshold's rate, apparently the improvement only shows up in the constant.  But compare this to Theorem~2.1 in \citet{pas.szabo.vaart.uq}, where they only have negative results about the coverage probability for marginal credible sets under the horseshoe prior when the non-zero means are smaller than a threshold like in \eqref{eq:bound}, even with a so-called ``blow-up factor.''  But, most importantly, we show numerically in Section~\ref{S:examples}, that this difference between the horseshoe and our empirical priors shows up in finite-samples.

\subsection{Extensions}

There are two immediate directions in which one might like to generalized the above results: to the regression setting, where the mean vector is $\theta = X\beta$, with $X$ a $n \times p$ matrix of covariates and $\beta$ a sparse $p$-vector of coefficients, possibly with $p \gg n$, and/or the error variance, $\sigma^2$, is unknown.  For the former case, \citet{martin.mess.walker.eb} establish posterior concentration rates for a version of the empirical prior (on $\beta$) considered here, along with selection consistency.  For this regression problem, the key technical obstacle is separating $\beta$ from $X\beta$, which requires use of a so-called ``sparse singular value'' of $X$; see Equation~(13) in \citet{martin.mess.walker.eb}, which is very similar to Equation~(11) in \citet{ariascastro.lounici.2014} and Definition~2.3 in \citet{castillo.schmidt.vaart.reg}.  These results for the regression problem where used by \citet{ebpred}---much like we used Theorems~\ref{thm:rate}--\ref{thm:selection} here---to establish a Bernstein--von Mises theorem and to investigate coverage properties of marginal credible intervals, in particular, for the purpose of out-of-sample prediction.  For the second direction, unknown $\sigma^2$, so far there has been no theoretical investigation into the performance of these posteriors based on an empirical prior.  The simulations presented in \citet{ebpred}, using a either a plug-in estimator of or a prior for $\sigma^2$, show good empirical performance.  So we fully expect that the theoretical properties of the posterior distribution would not be affected by unknown $\sigma^2$ either, but this still needs to be worked out.  We believe that this is doable because, at least with a conjugate prior for $\sigma^2$, there would be a closed-form expression for the marginal posterior distribution of $S$, which is what our theoretical analysis here focused on.  So, extending the results here would boil down to analyzing a different version of the ``$\pi^n(S)$'' expression considered here.

\section{Numerical investigations}
\label{S:examples}

In this section we investigate the numerical performance of the posterior distributions based on the two styles of empirical priors described in Section~\ref{SS:prior} compared to those based on the horseshoe prior introduced in \citet{carvalho.polson.scott.2010}.  In particular, we consider the following three methods:
\begin{itemize}
\item[HS.] For the horseshoe method we employ the {\tt horseshoe} package in R \citep{horseshoe.package} where maximum marginal likelihood, based on their function {\tt HS.MMLE}, is used to estimate the global scale parameter.  
\item[EB1.] The approach in \citet{martin.walker.eb} based on the beta--binomial prior $f_n$.  We use exactly the settings in that paper and the R code they provided.  
\item[EB2.] The approach in \citet{martin.mess.walker.eb} based on the complexity prior $f_n$.  We use the R code they provided, specialized from the regression to the means problem, and we take $\alpha = 0.95$ and $\tau = 0.025$, so that $\alpha + \tau < 1$; see Remark~\ref{re:alpha}.  
\end{itemize} 
Here we assume that $\sigma^2=1$ and all three methods use this known value.  

Our goal here is to compare the performance of marginal credible intervals for the three approaches described above.  The specific setting we consider here, similar to that in Section~2 of \citet{pas.szabo.vaart.uq}, assumes the first five means are relatively large, namely, $\theta_1^\star = \cdots = \theta_5^\star = 7$, the second five means are of intermediate magnitude, namely, $\theta_6^\star = \cdots = \theta_{10}^\star = 2$, $\theta_{11}^\star$ will vary, and the remaining $\theta_{12}^\star,\cdots,\theta_n^\star$ are 0.  We want to know how large does $\theta_{11}^\star$ have to be in order for the (two-sided) marginal credible interval for $\theta_{11}$ to attain the nominal frequentist coverage probability, which we take as 95\%, i.e., $\gamma=0.05$.  Figure~\ref{fig:cvg} plots the empirical coverage probability and average lengths of the marginal credible intervals for $\theta_{11}$, as a function of the signal size $\theta_{11}^\star$, for two different values of $n$, namely, $n=200$ and $n=500$.  These are Monte Carlo estimates based on 500 data sets at each value of $\theta_{11}^\star$ along a grid from 0 to 10. 

\begin{figure}[t]
\begin{center}
\subfigure[Coverage probability, $n=200$]{\scalebox{0.6}{\includegraphics{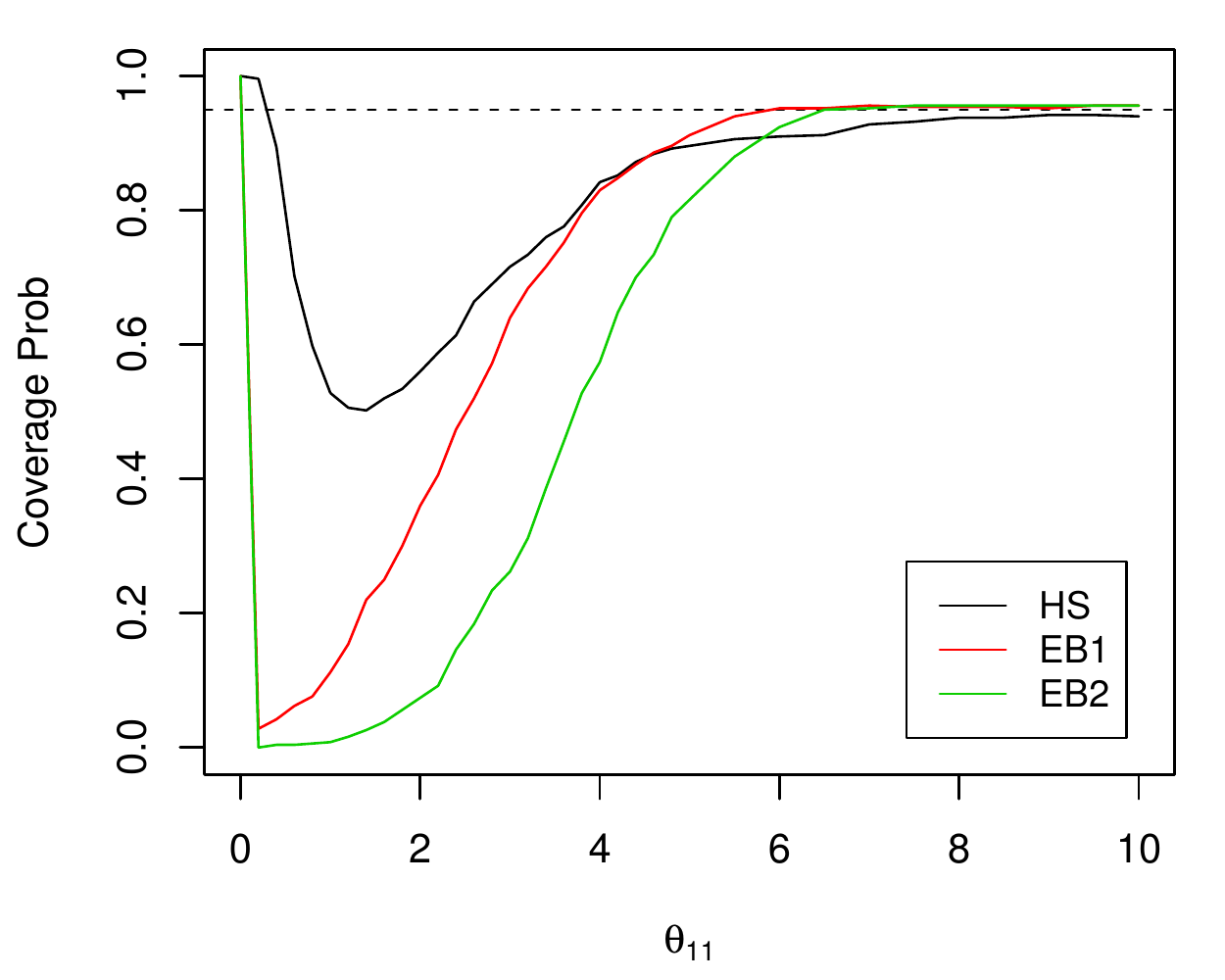}}}
\subfigure[Mean length, $n=200$]{\scalebox{0.6}{\includegraphics{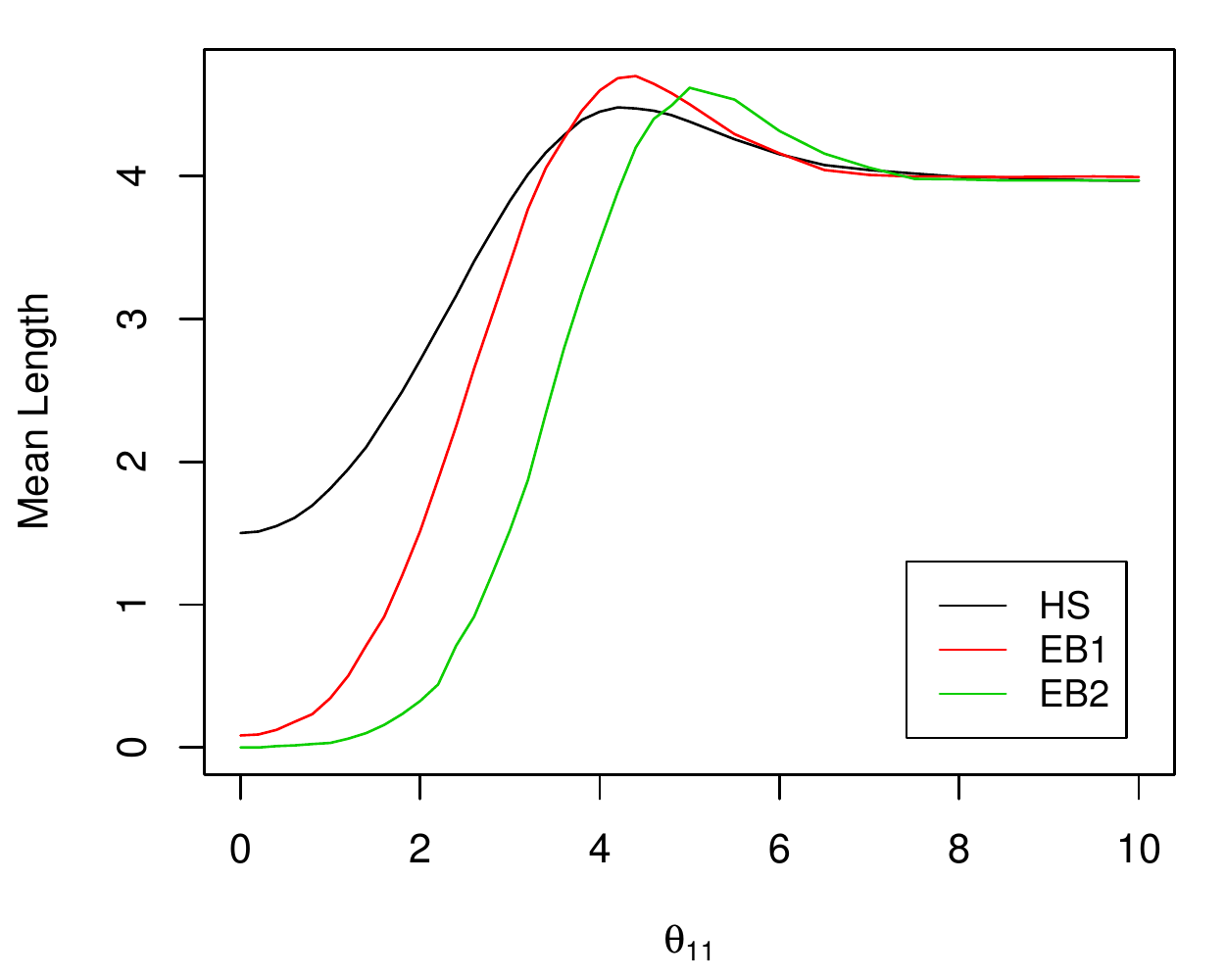}}}
\subfigure[Coverage probability, $n=500$]{\scalebox{0.6}{\includegraphics{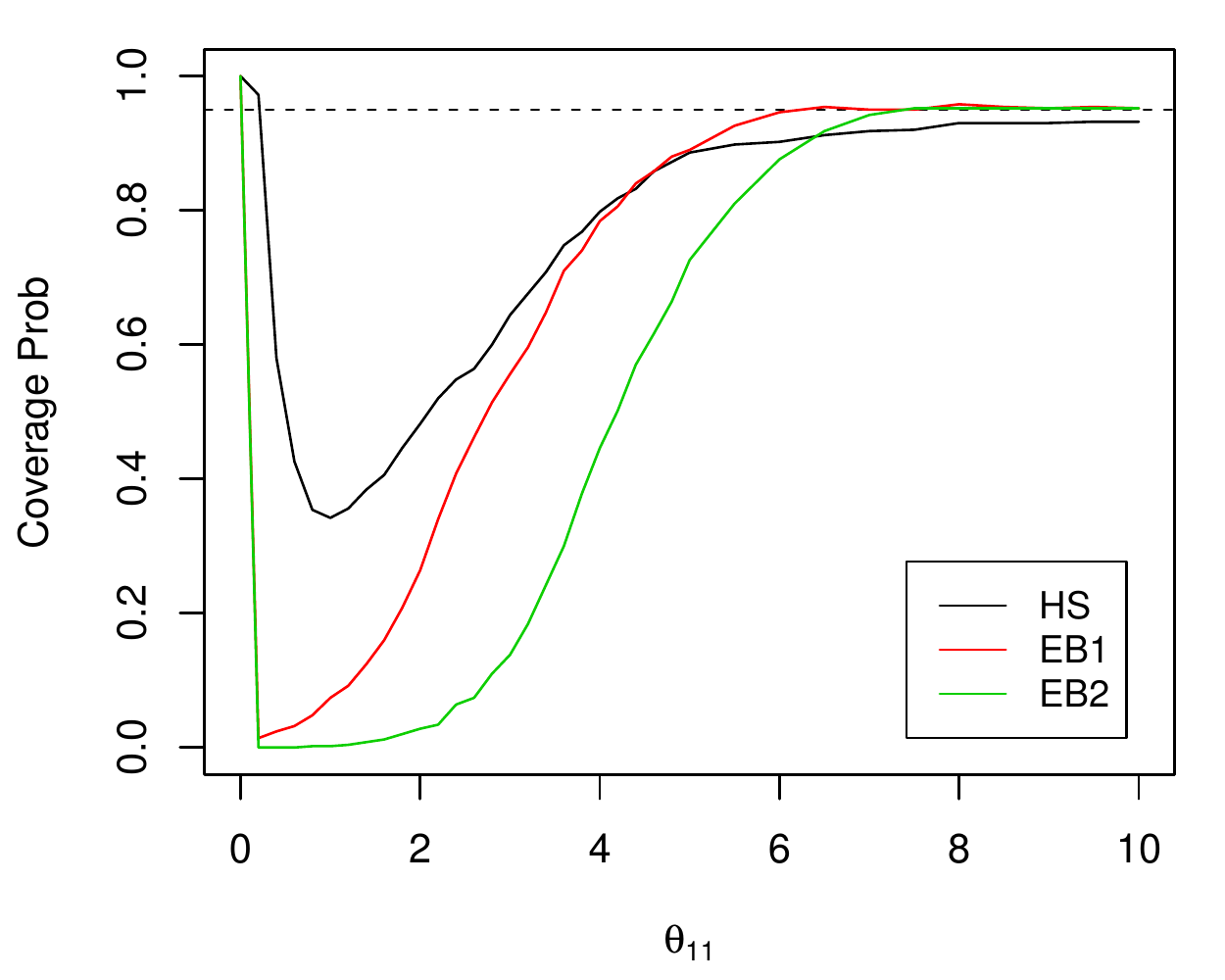}}}
\subfigure[Mean length, $n=500$]{\scalebox{0.6}{\includegraphics{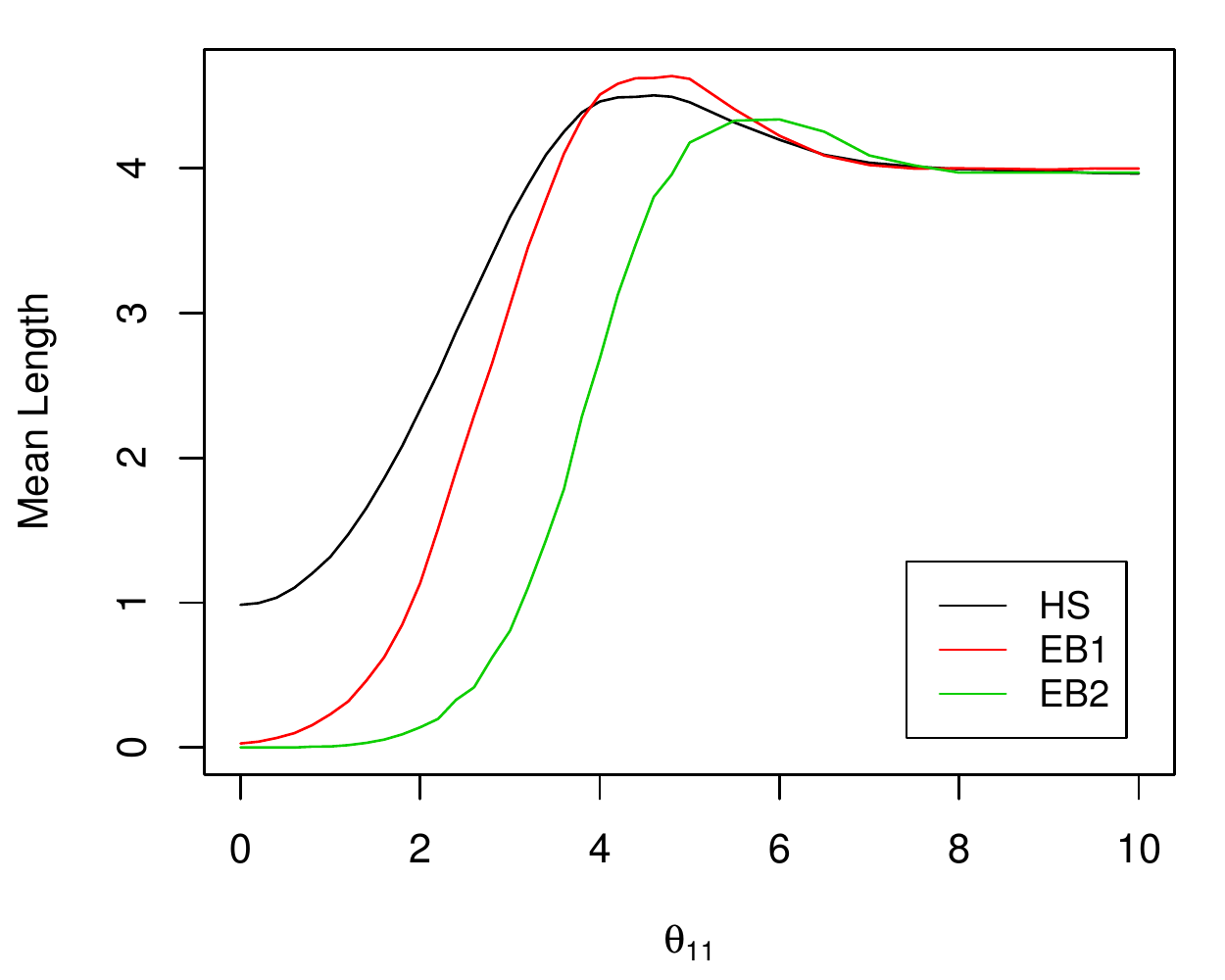}}}
\end{center}
\caption{Plots of the coverage probability and mean length of the marginal posterior credible intervals for $\theta_{11}$, as a function of $\theta_{11}^\star$, based on the three methods---HS, EB1, and EB2---and for two different sample sizes.}
\label{fig:cvg}
\end{figure}

The primary question we seek to answer here is {\em for what range of $\theta_{11}^\star$ values is the nominal 95\% coverage rate attained?}  The theoretical results in Section~\ref{S:coverage} indicated that the credible sets based on the empirical prior might have an advantage over the horseshoe-based credible sets in this respect.  So our conjecture is that the former will have a slightly broader range of nominal-coverage-attainment than the latter.  The results in Figure~\ref{fig:cvg} confirm this conjecture, that is, both the EB1 and EB2 coverage curves reach the nominal level well before HS for $n=200$ and $n=500$.  An interesting surprise is that EB1, based on the beta--binomial prior in \citet{martin.walker.eb} performs considerably better in terms of coverage than EB2 based on the complexity prior.  In our opinion, the Gibbs sampler for EB1 is a more elegant approach than the Metropolis--Hastings implementation of EB2, but it is not clear why the former would outperform the latter in terms of coverage probabilities.  And the length curves in Panels~(b) and (d) indicate that the EB gains over HS in coverage are not due to excessive length.  

Finally, despite being based on a two-groups or spike-and-slab formulation, both EB implementations were much faster to compute than HS based on the {\tt horseshoe} package.  In fact, we also ran the above experiment for $n=1000$ and could easily produce results for EB1 and EB2, which look similar to those in Figure~\ref{fig:cvg}, but the HS computations were prohibitively slow and regularly produced errors, hence not reported here.

\section{Conclusion}
\label{S:discuss}

In this paper, we focus on uncertainty quantification derived from posterior distributions based on a class of empirical priors.  This is certainly a relevant question to ask about such methods, given their {\em double-use} of the data.  That is, while estimation may not be negatively affected by the use of data in both the prior and the likelihood, it is possible that such an ``informative'' prior would make the posterior spread too narrow and, consequently, the coverage probability would fall below the nominal level.  Our investigation here revealed that this actually is not the case; in fact, in a certain sense, the empirical prior leads to improved uncertainty quantification compared to a more traditional approach where the prior is free of data.  Follow-up investigations should look into other aspects of uncertainty quantification, such as credible balls for the full $\theta$ vector with respect to various distances, and extension of these results to the regression setting and other structured high-dimensional problems.  

A unique feature of the empirical priors in this normal mean model is that the conditional prior variance for $\theta_S$, given $S$, is constant, it does not depend on data or even the sample size.  In general, however, the empirical priors constructed in \citet{martin.walker.deb} would have prior spread vanishing at some rate with $n$.  For example, in the regression setting, after writing the coefficient vector as $\beta=(S, \beta_S)$, the conditional empirical prior for $\beta_S$, given $S$, is $\nm_{|S|}(\hat\beta_S, \sigma^2 \tau^{-1} (X_S^\top X_S)^{-1})$, where $X_S$ is the $n \times |S|$ sub-matrix of the full $n \times p$ matrix $X$, and $\hat\beta_S$ is the least-squares estimate corresponding to $X_S$.  Note that each entry in this variance is $O(n^{-1})$.  Despite the relatively tight prior concentration, the method in \citet{martin.mess.walker.eb} performs well in terms of estimation and variable selection.  That the prior variance is consistent with that of the sampling distribution of $\hat\beta_S$, the prior mean, suggests that uncertainty quantification will not be impacted by the prior concentration.  Indeed, in \citet{ebpred}, the focus was on posterior credible sets for out-of-sample prediction in the sparse high-dimensional regression problem and, in addition to the theoretical results, their simulation study reveals properties similar those demonstrated here, namely, that the empirical Bayes credible sets achieve the desired coverage and are no less efficient than the horseshoe, even with unknown error variance.  Moreover, in a monotone density estimation context, \citet{ebmono} constructed an empirical prior, whose spread was vanishing with $n$, and his simulation experiments revealed that the coverage of posterior credible sets was not negatively affected by the prior concentration and, in fact, the coverage performance was better in some cases than the proper Bayesian posterior credible sets from the Dirichlet process-based formulation in \citet{salomond2014}.  Full justification of these claims in this and other nonparametric problems, however, requires further numerical and theoretical investigations.

\section*{Acknowledgments}

The authors thank the editors of the special issue of {\em Sankhya A} dedicated to Jayanta K.~Ghosh for the invitation to contribute, and the anonymous reviewers for their helpful suggestions that improved both our results and presentation.  This work is partially supported by the National Science Foundation, DMS--1737933.

\appendix

\section{Proof of Theorem~\ref{thm:selection}}
\label{S:proofs}

The proof strategy here closely follows that of Theorem~$5'$ in the supplement to \citet{martin.mess.walker.eb} presented in the most recent arXiv version ({\tt arXiv:1406.7718}).  To fix notation, let $S^\star$ be the true configuration of size $s^\star=|S^\star|$, and let $S^\dagger \subseteq S^\star$ be the set of all $i$ such that $|\theta_i^\star| \geq \rho_n$, where $\rho_n$ is as in \eqref{eq:bound}, and write $s^\dagger=|S^\dagger|$.  

Based on Theorem~2 in \citet{martin.mess.walker.eb}, we can restrict to configurations $S$ such that $|S| \leq Cs^\star$, where $s^\star = |S^\star|$ and $C$ is a large constant.  Take such an $S$ that also satisfies $S \not\supseteq S^\dagger$.   Then $\pi^n(S)$ can be bounded as follows:
\[ \pi^n(S) \leq \frac{\pi^n(S)}{\pi^n(S^\dagger)} = \frac{\pi(S)}{\pi(S^\dagger)} z^{|S|-s^\dagger} e^{\frac{\alpha}{2\sigma^2}\{\|Y_{S^{\dagger c}}\|^2 - \|Y_{S^c}\|^2\}}, \]
where $z = (1 + \alpha \tau^{-1})^{-1/2} < 1$.  A key observation is that 
\[ \|Y_{S^{\dagger c}}\|^2 - \|Y_{S^c}\|^2 = \|Y_{S \cap S^{\dagger c}}\|^2 - \|Y_{S^c \cap S^\dagger}\|^2, \]
and the latter two terms are independent since they depend on disjoint sets of $Y_i$'s.  Therefore, using this independence and the familiar central and non-central chi-square moment generating functions, we get 
\[ \E_{\theta^\star} e^{\frac{\alpha}{2\sigma^2}\{\|Y_{S^{\dagger c}}\|^2 - \|Y_{S^c}\|^2\}} = (1-\alpha)^{-|S \cap S^{\dagger c}|} (1+\alpha)^{-|S^c \cap S^\dagger|} e^{-\frac{\alpha}{2(1+\alpha)\sigma^2} \|\theta_{S^c \cap S^\dagger}^\star\|^2}. \]
By definition of $S^\dagger$, and the fact that $1+\alpha > 1$, the above expectation can be upper-bounded by 
\[ (1-\alpha)^{-|S \cap S^{\dagger c}|} (n^M)^{-|S^c \cap S^\dagger|}. \]
Putting the pieces together we have 
\[ \E_{\theta^\star} \pi^n(S) \leq \frac{\pi(S)}{\pi(S^\dagger)} z^{|S|-s^\dagger} (1-\alpha)^{-|S \cap S^{\dagger c}|} (n^M)^{-|S^c \cap S^\dagger|}. \]
We want to sum this over all $S \not\supseteq S^\dagger$ but, since it only involves size of $S$, we only need to sum over sizes.  Indeed, after plugging in the definition of $\pi(S)$ we get 
\[ \sum_{S: S \not\supseteq S^\dagger, |S| \leq C s^\star} \E_{\theta^\star} \pi^n(S) \leq \sum_{s=0}^{Cs^\star} \sum_{t=0}^{s \wedge s^\dagger} \frac{\binom{s}{t} \binom{n-s^\dagger}{s-t} \binom{n}{s^\dagger}}{\binom{n}{s}} \frac{f_n(s)}{f_n(s^\dagger)} z^{s-s^\dagger} (1-\alpha)^{-(s-t)} (n^M)^{-(s^\dagger - t)}. \]
For the binomial coefficient ratio we have the following simplification and bound:
\[ \frac{\binom{s}{t} \binom{n-s^\dagger}{s-t} \binom{n}{s^\dagger}}{\binom{n}{s}} = \binom{s}{t} \binom{n-s}{s^\dagger - t} \leq s^{s-t} n^{s^\dagger - t}. \]
Next, to bound the double-sum, split it into two parts:
\[ \sum_{s=0}^{Cs^\star} \sum_{t=0}^{s \wedge s^\dagger} (\cdots) = \sum_{s=0}^{s^\dagger - 1} \sum_{t=0}^s (\cdots) + \sum_{s=s^\dagger}^{Cs^\star} \sum_{t=0}^{s^\dagger} (\cdots) \]
We need to show that both parts on the right-hand side above vanish as $n \to \infty$.  For the first double-sum we have 
\[ \sum_{s=0}^{s^\dagger - 1} \sum_{t=0}^s (\cdots) = \sum_{s=0}^{s^\dagger - 1} \Bigl( \frac{1}{K_1 n^{M-a_1-1}} \Bigr)^{s^\dagger - s} \sum_{t=0}^s \Bigl( \frac{s}{(1-\alpha)n^{M-1}} \Bigr)^{s-t}. \]
Since $M > 1 + a_1$, the inner sum is $O(1)$ and the outer sum---because there is a common $n^{-(M-a_1-1)}$ factor---is $o(1)$ as $n \to \infty$.  Similarly, for the second double-sum we have 
\[ \sum_{s=s^\dagger}^{Cs^\star} \sum_{t=0}^{s^\dagger} (\cdots) = \sum_{s=s^\dagger}^{Cs^\star} \Bigl( \frac{K_2 s}{(1-\alpha)n^{a_2}} \Bigr)^{s-s^\dagger} \sum_{t=0}^{s^\dagger} \Bigl( \frac{s}{1-\alpha} \Bigr)^{s^\dagger - t} \Bigl( \frac{1}{n^{M-1}} \Bigr)^{s^\dagger - t}. \]
The inner sum is $O(1)$ and, since $a_2 < a_1$, the outer sum is upper-bounded by $O(s^\star n^{-a_2})$ which goes to 0 by assumption.  Both terms in the double-sum above vanish as $n \to \infty$, thus proving the claim.  

\ifthenelse{1=1}{
\bibliographystyle{apalike}
\bibliography{/Users/rgmarti3/Dropbox/Research/mybib}
}{

}

\end{document}